\documentclass[leqno]{amsart}

\usepackage{stmaryrd,graphicx}
\usepackage{amssymb,mathrsfs,amsmath,amscd,amsthm}
\usepackage{float}
\usepackage{graphicx,float}
\usepackage[all,cmtip]{xy}
\usepackage{amsfonts,txfonts,pxfonts,latexsym,wasysym}
\author{Jeremy Brazas}

\address{Department of Mathematics; West Chester University; 25 University Ave.; West Chester, PA, 19383}

\title{On the discontinuity of the $\pi_1$-action}


\newtheorem{theorem}{Theorem}
\newtheorem{lemma}[theorem]{Lemma}
\newtheorem{proposition}[theorem]{Proposition}
\newtheorem{corollary}[theorem]{Corollary}
\theoremstyle{definition}

\theoremstyle{definition}
\theoremstyle{definition}
\theoremstyle{definition}\newtheorem{problem}[theorem]{Problem}
\theoremstyle{definition}\newtheorem{remark}[theorem]{Remark}

\newcommand{\wt}{\widetilde}

\newcommand{\bbr}{\mathbb{R}}

\newcommand{\bbh}{\mathbb{H}}

\newcommand{\bbz}{\mathbb{Z}}
\newcommand{\bbn}{\mathbb{N}}

\newcommand{\ui}{[0,1]}

\newcommand{\bfz}{\mathbf{0}}

\newcommand{\bspaces}{\mathbf{Top}_{\ast}}

\begin{document}
\begin{abstract}
We show the classical $\pi_1$-action on the $n$-th homotopy group can fail to be continuous for any $n$ when the homotopy groups are equipped with the natural quotient topology. In particular, we prove the action $\pi_1(X)\times \pi_n(X)\to\pi_n(X)$ fails to be continuous for a one-point union $X=A\vee \mathbb{H}_n$ where $A$ is an aspherical space such that $\pi_1(A)$ is a topological group and $\mathbb{H}_n$ is the $(n-1)$-connected, n-dimensional Hawaiian earring space $\mathbb{H}_n$ for which $\pi_n(\mathbb{H}_n)$ is a topological abelian group.
\end{abstract}
\maketitle

\section{Introduction}
For a based topological space $X$, let $\Omega^n(X)$ denote the space of based maps $S^n\to X$ with the compact-open topology. For $n\geq 1$, the n-th homotopy group $\pi_n(X)$ may be endowed with the natural quotient topology inherited from $\Omega^n(X)$ so that the map $q:\Omega^n(X)\to \pi_n(X)$, $q(f)=[f]$ identifying based homotopy classes is a topological quotient map.

If $X$ has the homotopy type of a CW-complex, then $\pi_n(X)$ is a discrete group for all $n\geq 1$ \cite{CMdiscrete,GHMMthg}. More generally, $\pi_n(X)$ can fail to be a topological group in any dimension. For instance, if $\bbh$ is the Hawaiian earring, then $\pi_1(\bbh)$ fails to be a topological group \cite{Fab10HE}. Other counterexamples for $n=1$ related to free topological groups appear in \cite{Br10.1}. The higher homotopy groups are slightly better behaved since if $\bbh_n$ is the n-dimensional Hawaiian earring, $\pi_n(\bbh_n)$, $n\geq 2$ is a topological group isomorphic to the direct product $\prod_{n=1}^{\infty}\bbz$ of discrete groups \cite{GHMMnhawaiian}. The improved behavior for $n\geq 2$ is essentially because, as suggested in \cite{BarMil}, these groups are commutative in an infinitary sense. Despite this case, examples in \cite{Fab11CG} show that multiplication in $\pi_n(X)$ can still fail to be continuous.

Although multiplication in homotopy groups can fail to be continuous, $\pi_n:\bspaces\to\mathbf{qTopGrp}$ is a homotopy invariant functor from the category of based spaces to the category of quasitopological groups\footnote{A \textit{quasitopological group} $G$ is a group equipped with a topology such that inversion is continuous and multiplication is continuous in each variable. See \cite{AT08} for a general theory of quasitopological groups.} (abelian when $n\geq 2$) and continuous group homomorphisms. The topological structure of $\pi_1(X)$ is studied in depth in \cite{Br10.1} and \cite{BFqtop}. 

In the current paper, we address the continuity of the natural action of $\pi_1(X)$ on $\pi_n(X)$. For each $n\geq 1$, we construct a compact metric space $X\subseteq \bbr^{n+1}$ such that the action $\pi_1(X)\times \pi_n(X)\to \pi_n(X)$ is not continuous. Our example is constructed as a one-point union $X=A\vee \bbh_n$ where:
\begin{itemize}
\item  $\pi_1(A)$ is a topological group and $A$ is aspherical, i.e. $\pi_n(A)=0$ for $n\geq 2$,
\item $\bbh_n$ is $(n-1)$-connected and $\pi_n(\bbh_n)$ is a topological abelian group.
\end{itemize}

A fruitful approach to giving the homotopy groups a topology that makes them genuine topological groups is to apply a reflection functor $\tau$ to $\pi_n(X)$, which deletes the smallest set of open sets from the quotient topology so that one obtains continuity of multiplication. See \cite{Braztopgrp} for more on this construction where the resulting topological groups are denoted $\pi_{n}^{\tau}(X)$. For $n=1$, this topology has helped to prove new results in topological group theory \cite{BrazOpenSub}. An analogous construction has been carried out for the universal path space in \cite{VZtautop}. Since writing the paper \cite{Braztopgrp}, the author has been unable to confirm the continuity of the $\pi_1$-action for these ``topological homotopy groups." Since these groups are topological groups, the case $n=1$ is trivial since the $\pi_1$-action is just the action by conjugation and continuity follows from the continuity of the group operations.

\begin{problem}
For any $n>1$ and space $X$, must the natural action $\pi_{1}^{\tau}(X)\times \pi_{n}^{\tau}(X)\to \pi_{n}^{\tau}(X)$ be continuous?
\end{problem}

\noindent\textbf{Acknowledgements.} The author is grateful to Paul Fabel for helpful conversations related to this topic.

\section{Separate continuity of the $\pi_1$-action}\label{quotientsection}

Given based topological spaces $(X,x_0)$ and $(Y,y_0)$, let $Y^X$ denote the space of basepoint-preserving maps $f:(X,x_0)\to (Y,y_0)$ with the compact-open topology generated by subbasic sets $\langle K,U\rangle =\{f\in Y^X|f(K)\subseteq U\}$ for $K\subseteq X$ compact and $U\subseteq Y$ open. If $x_0\in A\subseteq X$, let $Y^{(X,A)}\subseteq Y^X$ be the subspace of maps $f$ for which $f(A)=y_0$. Let $[X,Y]$ be the set of basepoint-preserving homotopy classes of maps equipped with the quotient topology with respect to the map $Y^X\to [X,Y]$, $f\mapsto [f]$ that identifies homotopy classes. In particular, $Y^{S^n}=\Omega^n(Y)$ and $\pi_n(Y)=[S^n,Y]$. For $\alpha_1,\alpha_2,...,\alpha_k\in \Omega^n(Y)$, $\alpha_1\cdot \alpha_2\cdots \alpha_k$ denotes the map $S^n\to Y$ which is the standard k-fold concatenation representing the product $[\alpha_1][\alpha_2]\cdots [\alpha_n]$ in $\pi_n(Y)$.

We recall the definition of the $\pi_1$-action as given in \cite{WhiteheadEOH}. Suppose the space $X$ with basepoint $x_0\in X$ is well-pointed, i.e. $\{x_0\}\to X$ is a cofibration. Set $Z=I\times \{x_0\}\cup X\times \{0\}$ where $I=[0,1]$ is the unit interval. The condition that $(X,x_0)$ is well-pointed is equivalent to the existence of a retraction $r:X\times I\to Z$. Fixing such a retraction, let $i_1:X\to X\times \{1\}$, $i_1(x)=(x,1)$ be the inclusion. 

For the moment, we identify $\Omega(Y)$ with the mapping space $Y^{(I,\{0,1\})}$. Two maps $\alpha\in \Omega(Y)$ and $\gamma\in Y^X$ uniquely determine a map $f:Z\to Y$ by $f(t,x_0)=\alpha^{-}(t)$ and $f(x,0)=\gamma(x)$. Let $\alpha\ast\gamma=f\circ r\circ i:X\to Y$ and note this is a based map. This construction is well-defined on homotopy classes in the sense that $\pi_1(Y)\times [X,Y]\to [X,Y]$, $([\alpha],[\gamma])\mapsto [\alpha\ast\gamma]$ is a group action on the set $[X,Y]$.

\begin{proposition}\label{jointcont}
The map $\Omega(Y)\times Y^X\to Y^X$, $(\alpha,\gamma)\mapsto \alpha\ast\gamma$ is continuous and the group action $\pi_1(Y)\times [X,Y]\to [X,Y]$ is continuous in each variable.
\end{proposition}

\begin{proof}
Let $A=\{(x_0,0),(x_0,1)\}\subseteq X\times I$. Since there is a canonical homeomorphisms $Y^{S^1}\times Y^X \cong Y^{S^1\vee X}\cong Y^{(Z,A)}$, $(\alpha,\gamma)\mapsto f$, it suffices check the map $Y^{(Z,A)}\to Y^X$, $f\mapsto f\circ r\circ i$ is continuous. It is a fact from elementary topology that if $g:X\to X'$ is a map, then the induced function $Y^{X'}\to Y^X$, $f\mapsto  f\circ g$ is continuous. We apply this fact to the following maps:
\begin{enumerate}
\item the map $r:X\times I\to Z$ induces a map $Y^{Z}\to Y^{X\times I}$ which restricts to a well-defined (since $r$ is a retraction) continuous map $Y^{(Z,A)}\to Y^{(X\times I,A)}$ on subspaces,
\item the map $i_1:X\to X\times I$, induces a map $Y^{X\times I}\to Y^X$ which restricts to a map $Y^{(X\times I,A)}\to Y^X$.
\end{enumerate}
The composition of these continuous maps agrees with the function $(\alpha,\gamma)\mapsto \alpha\ast\gamma$. Therefore, it is continuous.

For the group action, consider the commuting diagram
\[\xymatrix{
\Omega(Y)\times Y^X \ar[d]_-{q_1\times q_2} \ar[r]^-{\ast} & Y^X \ar[d]^-{q_2}\\
\pi_1(Y)\times [X,Y] \ar[r] & [X,Y]
}\]
where the top map is the continuous action of loops and where $q_1$ and $q_2$ are the respective quotient maps. Although $q_1\times q_2$ is not necessarily a quotient map, if we fix $\alpha\in \Omega(Y)$ and replace $\Omega(Y)$ and $\pi_1(Y)$ with the one-point subspaces $\{\alpha\}$ and $\{[\alpha]\}$ respectively, then the corresponding restriction of $q_1\times q_2$ is a quotient map (since it is just $q_2$ times a trivial map). Since the top composition is continuous, it follows from the universal property of quotient spaces that the bottom map $\{[\alpha]\}\times [X,Y]\to [X,Y]$, which is the restriction of the group action, is continuous. Fixing $\gamma$ in a similar fashion shows that the restriction $\pi_1(Y)\times \{[\gamma]\}\to [X,Y]$ is continuous. Thus the action is continuous in each variable.
\end{proof}

\begin{corollary}
If $\pi_1(X)$ or $\pi_n(X)$ is discrete, then the $\pi_1$-action $\pi_1(X)\times \pi_n(X)\to \pi_n(X)$ is continuous.
\end{corollary}

\begin{corollary}
Suppose $X$ is a metric space and $\pi_1(X)$ or $\pi_n(X)$ are first countable, then the $\pi_1$-action $\pi_1(X)\times \pi_n(X)\to \pi_n(X)$ is continuous.
\end{corollary}

\begin{proof}
Since $X$ is metric, $\Omega(X)\times \Omega^n(X)$ is first countable. If $\pi_1(X)\times \pi_n(X)$ is first countable, then according to the general theory of $k$-spaces \cite{Brown06}, the product of quotient maps $q_1\times q_2:\Omega(X)\times \Omega^n(X)\to \pi_1(X)\times \pi_n(X)$ is a quotient map. Applying the argument used in the second half of Proposition \ref{jointcont}, the continuity of the action follows from the continuity of $\ast:\Omega(X)\times \Omega^n(X)\to \Omega^n(X)$ and the universal property of quotient maps.
\end{proof}

\section{The space $A$}

Let $D\subseteq \bbr^2$ be the closed disk having $[0,1]\times \{0\}$ as a diameter. For each $m\geq 1$, let $C_m\subseteq \bbr^2$ be the circle having $[0,\frac{m+1}{m}]\times \{0\}$ as a diameter. We write $C_{\infty}$ for the boundary of $D$ and note $C_m\to C_{\infty}$ in the Hausdorff metric. Let $e_m=(\frac{m+1}{m},0)\in C_m$ and $e_{\infty}=(1,0)\in C_{\infty}$.

Define $A=D\cup \bigcup_{m\geq 1}C_m$. We denote the origin $(0,0)$ by $\bfz$ and take this point to be the basepoint of $A$. Observe that $A$ is compact but not locally path-connected at any point of $C_{\infty}\backslash \bfz$. For $m\in\{1,2,...,\infty\}$, let $\ell_m$ be a simple closed curve traversing $C_m$ once in the counterclockwise direction such that $\ell_m\to \ell_{\infty}$ uniformly. Set $x_m=[\ell_m]\in \pi_1(A)$. Since $\ell_m\to\ell_{\infty}$, $\ell_{\infty}$ is null-homotopic, and $\Omega(A)\to \pi_1(A)$ is continuous, $x_m$ converges to the trivial element $1\in \pi_1(A)$.

\begin{figure}[H]
\centering \includegraphics[height=2in]{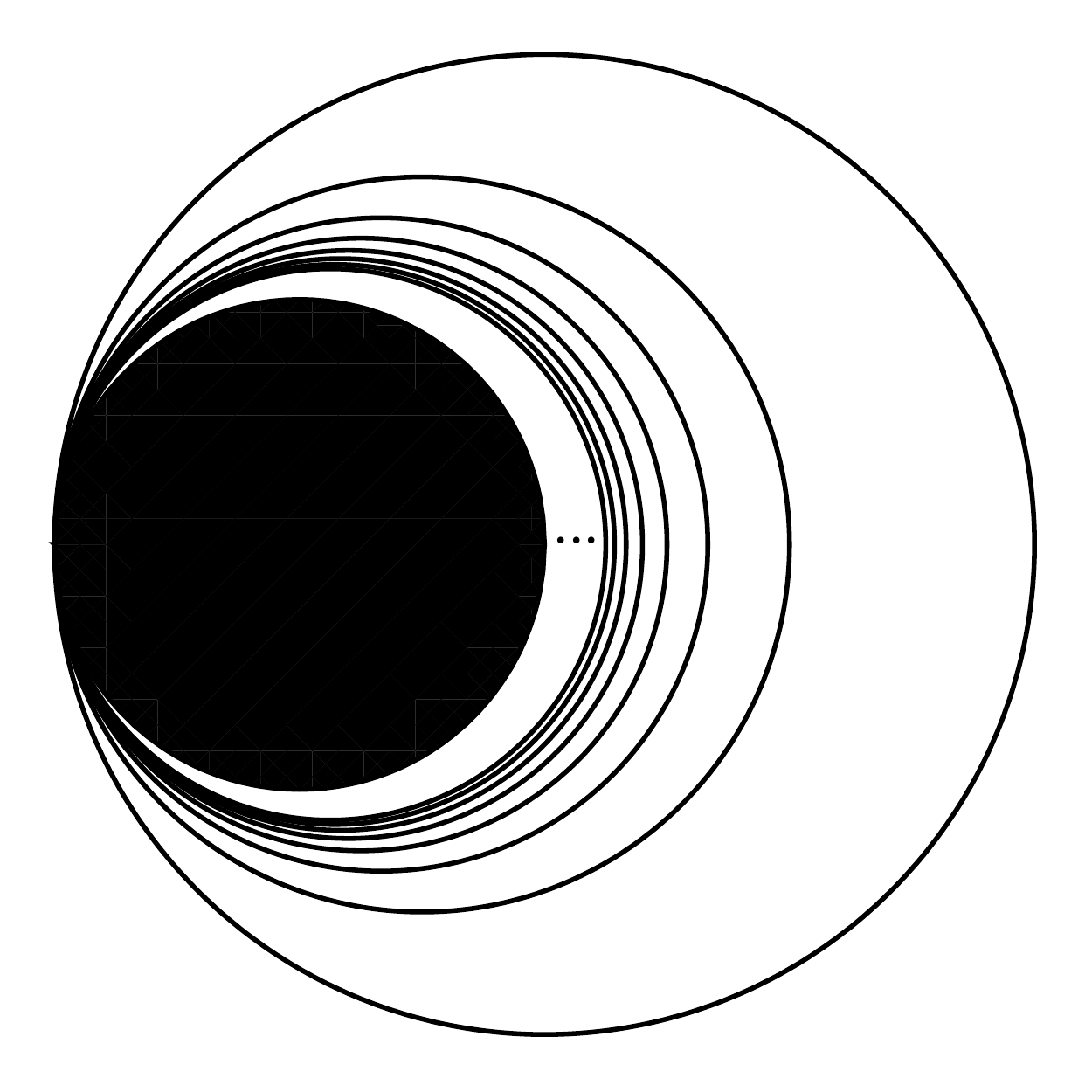}
\caption{\label{spacea}The space $A\subseteq \bbr^2$.}
\end{figure}

\begin{remark}
The space $A$ is aspherical, i.e. $\pi_n(A)=0$ for $n\geq 2$. To see this, one could apply general results about planar sets such as those in \cite{CCZAspherical}, however, it is simpler to utilize the locally path-connected coreflection $lpc(A)$ which has the same underlying set as $A$ and topology generated by the path components of open sets in $A$. In general, $X$ and $lpc(X)$ have naturally isomorphic homotopy and homology groups. In particular, since $lpc(A)$ is homotopy equivalent to a wedge of circles, $lpc(A)$ is aspherical and it follows that $A$ is aspherical.
\end{remark}

\begin{lemma}\label{finitelymanlemma}
Suppose $z_m\to z$ is a convergent sequence in a Hausdorff space $Z$ such that there exists an open neighborhood $U$ of $z$ so that none of the points $z_m$ lie in the path component of $U$ containing $z$. If $K$ is compact and locally path-connected and $f:K\to Z$ is a map, then $z_m\in f(K)$ for at most finitely many $m$.
\end{lemma}

\begin{proof}
Suppose $f(k_n)=z_{m_n}$ for $m_1<m_2<...$. Since $K$ is compact, we may replace $k_n$ with a convergent subsequence. Hence $k_n\to k$ in $K$ where we must have $f(k)=z$. Consider the neighborhood $U$ of $z$ as in the statement of the lemma. Find a path-connected neighborhood $V$ of $k$ in $K$ such that $f(V)\subseteq U$. Now infinitely many $k_n$ lie in $V$ and so infinitely many $z_{m_n}$ must lie in the same path component of $U$ as $z$; a contradiction.
\end{proof}

Certainly, $A$ enjoys the hypothesis of the previous lemma as applied to the convergent sequence $e_m\to e_{\infty}$. We now show that similar results exist for convergent sequences of maps.

\begin{lemma}\label{boundinglemmaforpaths}
Suppose $\alpha_n:[0,1]\to A$ is a sequence of maps such that $|\{e_1,e_2,...\}\cap \alpha_n([0,1])|\to \infty$. Then there is no map $\alpha:[0,1]\to A$ for which $\alpha_n\to \alpha$ uniformly.
\end{lemma}

\begin{proof}
Suppose $\alpha_n$ satisfies the hypothesis of the lemma. Without loss of generality, we may assume $|\{e_1,e_2,...\}\cap \alpha_n([0,1])|\geq n$. The open set $U=A\backslash \bfz$ has open path components $D\backslash \bfz$ and $C_m\backslash \bfz$, $m\geq 1$. For the moment, fix $n\geq 1$. The open set $\alpha_{n}^{-1}(U)$ is a countable disjoint union of open intervals. By assumption, we may find components $(a_1,b_1)$,...,$(a_n,b_n)$ of $\alpha_{n}^{-1}(U)$, naturally ordered within $\ui$, and points $t_{2k}\in (a_k,b_k)$ such that $\alpha(t_{2k})=e_{m_k}$ for $n$-distinct integers $m_1,m_2,...,m_n$. Let $t_{2k-1}=a_{k}$ for $k=1,...,n$ and $t_{2n+1}=b_n$. Now we have $t_0<t_1<t_2<...<t_{2n+1}$ where $f(t_{2k-1})=\bfz$ and $\alpha(t_{2k})=e_{m_k}$. In particular, $\|\alpha_n(t_j)-\alpha_n(t_{j+1})\|>1$ for all $n\geq 1$, $j=0,1,...,2n$.

Suppose, to obtain a contradiction, that $\alpha_n\to \alpha$ uniformly for some $\alpha:[0,1]\to X$. Since the set $\{\alpha_n|n\geq 1\}$ is uniformly equicontinuous, there is a $\delta>0$ such that $\|\alpha_n(s)-\alpha_n(t)\|<1$ for all $n\geq 1$ and $s,t\in \ui$ with $|s-t|<\delta$. Find $n$ such that $\frac{1}{n}<\delta$ and recall the previous paragraph applied to $\alpha_n$. Since the intervals $(a_k,b_k)$, $k=1,...,n$ are disjoint, at least one of them, call it $(a_{k_0},b_{k_0})$, must have diameter less than $\frac{1}{n}$. Since $t_{2k_0}\in (a_{k_0},b_{k_0})= (t_{2k_0-1},b_{k_0})$, we have $|t_{2k_{0}-1}-t_{2k_{0}}|<\frac{1}{n}<\delta$ and thus $\|\alpha_n(t_{2k_{0}-1})-\alpha_n(t_{2k_{0}})\|<1$; a contradiction of the last sentence in the previous paragraph.
\end{proof}

\begin{lemma}\label{boundinglemma}
Suppose $K$ is a Peano continuum and $f_n:K\to A$ is a sequence of maps such that $|\{e_1,e_2,...\}\cap f_n(K)|\to \infty$. Then there is no map $f:K\to A$ for which $f_n\to f$ uniformly.
\end{lemma}

\begin{proof}
Suppose $f_n$ satisfies the hypothesis of the lemma. Since $K$ is a Peano continuum, there is an onto map $\beta:[0,1]\to K$. Let $\alpha_n=f_n\circ \beta$. If $f_n\to f$ uniformly, then $\alpha_n\to\alpha=f\circ \beta$ uniformly and $|\{e_1,e_2,...\}\cap \alpha_n([0,1])|\to \infty$; a contradiction of Lemma \ref{boundinglemmaforpaths}.
\end{proof}

\begin{proposition}\label{atopgrp}
$\pi_1(A)$ is freely generated by the set $\{x_1,x_2,...\}$ and is a Hausdorff topological group.
\end{proposition}

\begin{proof}
Note that $A_k=D\cup \bigcup_{1\leq m\leq k}C_m$ is a retract of $A$ by maps $\rho_k:A\to A_k$ collapsing $C_{k+1},C_{k+2},...$ homeomorphically onto $C_{\infty}$. The space $A_k$ is homotopy equivalent to the wedge of circles $\bigcup_{1\leq m\leq k}C_m$ so $\pi_1(A_k)$ is freely generated by the set $\{x_1,...,x_k\}$. The inclusions $A_1\to A_2\to...$ induce a directed system of discrete free groups $\pi_1(A_k)=F(x_1,...,x_k)$ such that $\varinjlim_{k}\pi_1(A_k)=\varinjlim_{k}F(x_1,...,x_k)=F(x_1,x_2,...)$. Moreover, the inclusions $A_n\to A$ induce a canonical homomorphism  $\phi:\varinjlim_{k}\pi_1(A_k)\to \pi_1(A)$. We check that $\phi$ is a group isomorphism. If $\alpha:S^1\to A$ is a loop, then by Lemma \ref{finitelymanlemma}, $\alpha$ has image in $A\backslash \{e_{M+1},e_{M+2},...\}$ for some $M$. Since $A\backslash \{e_{M+1},e_{M+2},...\}$ deformation retracts onto $A_M$, $\alpha$ is homotopic to a loop in $A_M$. This verifies the surjectivity of $\phi$. Suppose $\alpha:S^1\to A_k$ is a loop which contracts by a null-homotopy $h:S^1\times [0,1]\to A$. Once again applying Lemma \ref{finitelymanlemma}, $h$ must have image in $A\backslash \{e_{M+1},e_{M+2},...\}$ for some $M>k$. Composing $h$ with a deformation retract of $A\backslash \{e_{M+1},e_{M+2},...\}$ onto $A_M$, we obtain a null-homotopy of $\alpha$ within $A_M$. Thus $\alpha$ represents the trivial element in $\varinjlim_{k}\pi_1(A_k)$, proving injectivity.

A proof from first principles that $\pi_1(A)$ is a Hausdorff topological group is much more technical; we discuss the ingredients of the proof here, referring to more general work. Since $A$ is a planar set, $\pi_1(A)$ is Hausdorff according to \cite[Corollary 20]{BFqtop}. To see that $\pi_1(A)$ is a topological group, we call upon results in \cite{Br10.1}. The space $C_{\infty}\cup \bigcup_{n}C_n$ is precisely the space in \cite[Example 4.24]{Br10.1} whose fundamental group is shown to be the free Markov topological group generated by the subspace $\{x_1,x_2,...,x_{\infty}\}\subset \pi_1(A)$. It is known that if a 2-cell $e^2$ is attached to a path-connected space $Y$, then the inclusion $Y\to Y\cup e^2$ induces a homomorphism $\pi_1(Y)\to \pi_1(Y\cup e^2)$ which is a topological quotient map \cite[Lemma 4.4]{Brazasdiss}. Since $A$ is obtained from $C_{\infty}\cup \bigcup_{m}C_m$ by attaching a 2-cell with boundary $C_{\infty}$, $\pi_1(A)$ is the topological quotient group of the topological group $\pi_1(C_{\infty}\cup \bigcup_{m}C_m)$ by factoring the normal closure of $\{[\ell_{\infty}]\}$. Since the quotient of a topological group equipped with the quotient topology is a topological group, $\pi_1(A)$ is a topological group. 
\end{proof}

\begin{remark}
One can conclude from the previous proof that $\pi_1(A)$ is isomorphic to the free Graev topological group generated by the space $\{x_1,x_2,...,x_{\infty}\}$ with basepoint $x_{\infty}$.
\end{remark}

\section{The space $\bbh_n$}

Fix $n\geq 2$ and let $S_m\subseteq \bbr^{n+1}$ be the n-dimensional sphere centered at $(1/m,0,...,0)$ with radius $1/m$. Now $\bbh_{n}=\bigcup_{m\geq 1}S_m$ is the \textit{n-dimensional Hawaiian earring space} with basepoint $\bfz=(0,0,...,0)$. Let $y_m\in \pi_n(\bbh_n)$ be the homotopy class of the canonical homeomorphism $\gamma_m:S^n\to S_m\subset\bbh_n$ that generates $\pi_n(S_m)\cong \bbz$. 
\begin{figure}[H]
\centering \includegraphics[height=2.3in]{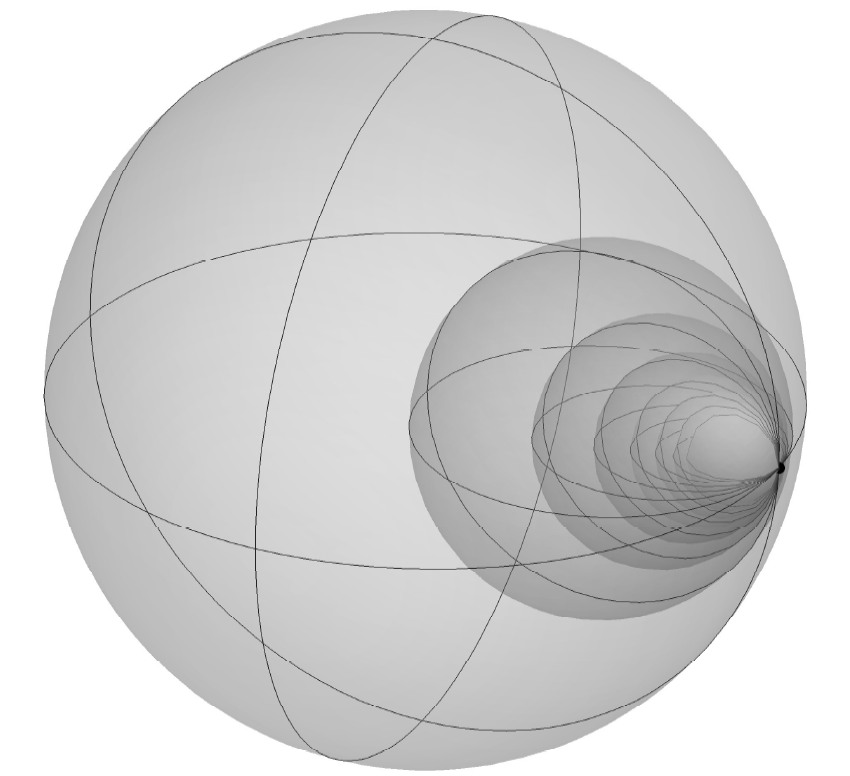}
\caption{\label{spacehe}The 2-dimensional Hawaiian earring $\bbh_2\subseteq \bbr^3$.}
\end{figure}

We note some technical achievements regarding $\bbh_n$:
\begin{enumerate}
\item \cite{EKndimhe} $\bbh_n$ is $(n-1)$-connected.
\item \cite{EKndimhe} If $r_M:\bbh_n\to S_M$ is the natural retraction collapsing $\bigcup_{m\neq M}S_m$ to $\bfz$, then the induced homomorphism $\pi_n(\bbh_n)\to \prod_{m\geq 1}\pi_n(S_m)$, $[\alpha]\mapsto ((r_1)_{\#}([\alpha]),(r_2)_{\#}([\alpha]),...)$ is an isomorphism. In particular, if $\gamma:S^n\to \bbh_n$ is a map such that $r_m\circ \gamma$ is null-homotopic for all $m$, then $\gamma$ is null-homotopic. Thus $\pi_n(\bbh_n)\cong \prod_{m\geq 1}\bbz$.
\item \cite{GHMMnhawaiian} The group isomorphism $\pi_n(\bbh_n)\to \prod_{m\geq 1}\bbz$ is an isomorphism of topological groups when $ \prod_{m\geq 1}\bbz$ is viewed as the the direct product of discrete groups.
\end{enumerate}

Although not explicitly stated in \cite{EKndimhe}, it follows that $\pi_n(\bbh_n)$ is isomorphic to the n-th singular homology group $H_n(\bbh_n)$. In particular, since $\bbh_n$ is $(n-1)$-connected, the Hurewicz Theorem implies that the Hurewicz homomorphism $h:\pi_n(\bbh_n)\to H_n(\bbh_n)$ is an isomorphism. One might initially be concerned about the applicability of the Hurewicz Theorem to spaces that are not CW-complexes, however, the theory of CW-approximations and weak homotopy equivalences (see p.352-357 of \cite{Hatcher}) permits the application.

\begin{remark}\label{finitecaseremark}
If $Y=\bigvee_{i=1}^{m}S^1\vee \bigvee_{j=1}^{k}S^n$ is a finite wedge of $1$-spheres and $n$-spheres, then the universal covering space $\wt{Y}$ may be viewed as the Caley graph $T$ with vertex set $F(x_1,...,x_m)$ with a copy of $\bigvee_{j=1}^{k}S^n$ attached at each vertex. By collapsing the tree to a point, one observes that $\pi_n(Y)\cong \pi_n(\wt{Y})\cong \bigoplus_{w\in F(x_1,...,x_m)}\pi_n(\bigvee_{j=1}^{k}S^n)\cong \bigoplus_{\bbn}\bbz$. In the following lemma, we seek to obtain an analogous result when we replace $\bigvee_{j=1}^{k}S^n$ with the infinite shrinking wedge $\bbh_n$, however, care must be taken since collapsing edges will not necessarily preserve homotopy type.
\end{remark}

\begin{lemma}\label{treecollapse}
Let $T$ be a finite tree. For each vertex $v$ of $T$, attach a copy of $\bbh_n$ to $T$ by the identification $v\sim \bfz$. Denote the resulting space $Z$. Then $Z$ is $(n-1)$-connected and the quotient map $\eta:Z\to Z/T$ induces a continuous group isomorphism $\eta_{\#}:\pi_n(Z)\to \pi_n(Z/T)$. In particular, $\pi_n(Z)\cong \prod_{v}\prod_{n\geq 1}\bbz\cong \prod_{\bbn}\bbz$.
\end{lemma}

\begin{proof}
The continuity of $\eta_{\#}$ follows from functorality of $\pi_n$. Our proof of the other claims is carried out by induction on the number of edges of $T$. The lemma follows from the results on $\bbh_n$ above when $T$ has a single vertex.

Suppose the lemma holds for all trees with $n$ vertices. Suppose $T$ has $n+1$ vertices. Let $T'\subset T$ be a subtree with $n$ edges and $E$ be the remaining edge of $T$ so that $T'\cup E=T$. For convenience, identify $E$ with $[0,1]$ so that $\{1\}=T'\cap E$ and the vertex $1$ is an end. Let $K_1,...,K_p$ be the copies of $\bbh_n$ attached to the vertices of $T'$ so that $Z'=T'\cup K_1\cup...\cup K_p$. Let $K_0$ denote the copy of $\bbh_n$ attached at $0$. By the induction hypothesis $Z'$ is $(n-1)$-connected, the quotient map $\eta ':Z'\to Z'/T'$ induces an isomorphism $(\eta ')_{\#}:\pi_n(Z')\to \pi_n(Z'/T')$. By choosing a homeomorphism $Z'/T'\cong \bbh_n$, we may identify $\pi_n(Z'/T')$ and $ H_n(Z'/T')$ with $\prod_{m\geq 1}\bbz$.

Since $Z$ is the one-point union of $(n-1)$-connected spaces $K_0\cup [0,1/2]$ and $[1/2,1]\cup Z'$, which are both locally contractible at $1/2$, $Z$ is also $(n-1)$-connected. Since both $Z$ and $Z'/T'$ are $(n-1)$-connected, it now suffices (by the Hurewicz Theorem) to show that $\eta_{\ast}:H_n(Z)\to H_n(Z/T)$ is an isomorphism.

Set $U=K_0\cup [0,2/3)$ and $V=(1/3,1]\cup Z'$. These open sets deformation retract onto $K_0$ and $Z'$ respectively. Since $U\cap V$ is contractible, the Mayer-Vietoris sequence for the triad $(Z,U,V)$ implies that the inclusions $i:K_0\to Z$ and $j:Z'\to Z$ induce an isomorphism $\phi=i_{\ast}-j_{\ast}: H_n(K_0)\oplus H_n(Z')\to H_n(Z)$.

The inclusions $k:K_0\to Z/T$, $l:Z'/T'\to Z/T$ induce a homomorphism $k_{\ast}-l_{\ast}:H_n(K_0)\oplus H_n(Z'/T')\to H_n(Z/T)$, which makes the bottom square in following diagram commute:
\[\xymatrix{
H_n(Z) \ar[r]^{\eta_{\ast}}  &  H_n(Z/T)\\
 H_n(K_0)\oplus H_n(Z')  \ar[u]^-{i_{\ast}-j_{\ast}} \ar[r]_-{id\oplus(\eta ')_{\ast}} & H_n(K_0)\oplus H_n(Z'/T') \ar[u]_-{k_{\ast}-l_{\ast}} 
}\]
To ensure the top map is an isomorphism it suffices to show that $k_{\ast}-l_{\ast}$ is an isomorphism. However, by identifying $H_n(K_0)\oplus H_n(Z'/T')$ with $\prod_{m\geq 1}\bbz\oplus \prod_{m\geq 1}\bbz$ and $H_n(Z/T)$ with $\prod_{m\geq 1}\bbz$, we see that $k_{\ast}-l_{\ast}:\prod_{m\geq 1}\bbz\oplus \prod_{m\geq 1}\bbz\to \prod_{m\geq 1}\bbz$ is a coordinate shuffling isomorphism determined by a homeomorphism $Z/T\cong \bbh_n$.
\end{proof}

%
\begin{corollary}\label{infinitetreecorollary}
Let $T$ be a tree with vertex set $V$ and basepoint $v_0\in V$ and $Z$ be the space $T$ with a copy $K_v$ of $\bbh_n$ attached at each $v\in V$. If $A_v$ is the arc in $T$ from $v_0$ to $v$, then the inclusion maps $A_v\cup K_v\to Z$ induce a group isomorphism $\rhoup:\bigoplus_{v\in V}\pi_n(A_v\cup K_v)\to \pi_{n}(Z)$. In particular, $\pi_n(Z)\cong \bigoplus_{|V|}\prod_{\bbn}\bbz$ as a group.
\end{corollary}

\begin{proof}
Lemma \ref{treecollapse} is precisely the case when $T$ is finite. When $T$ is infinite, $\rhoup$ is defined on elements as $\rhoup(\sum_{v}n_v[\delta_v])=\left[\prod_{n_v\neq 0}\delta_{v}^{n_v}\right]$ where all but finitely many $n_v\in\bbz$ are non-zero. Every map $S^n\to Z$ and null-homotopy of a map $S^n\to Z$ must have image in a subspace $Z_F\subset Z$ consisting of a finite subtree $F\subset T$ and all $K_v$ with $v\in F$. Hence the finite case directly implies the infinite case.
\end{proof}

\begin{remark}
It is well-known that the homomorphism $\eta_{\#}$ from Lemma \ref{treecollapse} fails to be surjective if we allow $n=1$.
\end{remark}

\section{The space $X=A\vee \bbh_n$}

We now fix $n\geq 1$ and consider the one-point union $X=A\vee \bbh_{n}$, which may be embedded as a compact subspace of $\bbr^{n+1}$. The main difficulty in what follows is ensuring that $\pi_1(X)$ and $\pi_n(X)$ are $T_1$.

\begin{figure}[H]
\centering \includegraphics[height=2in]{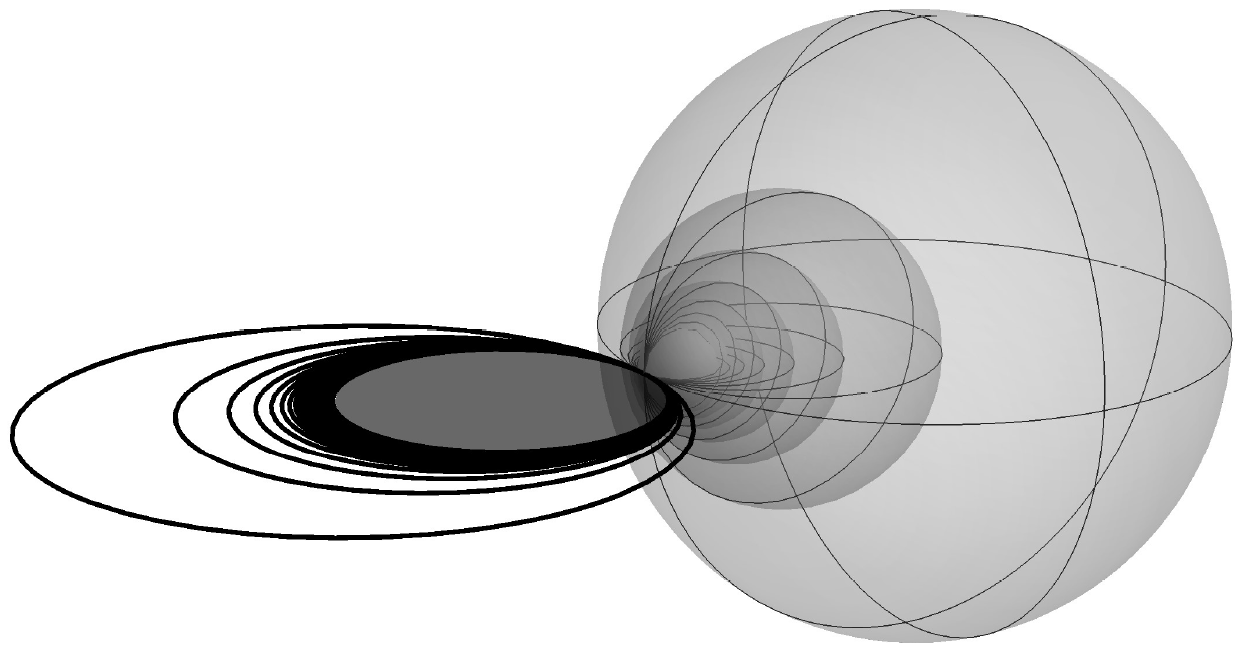}
\caption{\label{spacex}The space $X=A\vee \bbh_n$ for $n=2$.}
\end{figure}

\begin{proposition}
The inclusion $A\to X$ induces an isomorphism $\pi_1(A)\to \pi_1(X)$ of topological groups.
\end{proposition}

\begin{proof}
Since the space $\bbh_n$ is first countable and locally simply connected, it follows from a theorem of Griffiths \cite{Griffiths} that the inclusions $A\to X$ and $\bbh_n\to X$ induce an isomorphism $\pi_1(A)\ast \pi_1(\bbh_n)\to \pi_1(X)$ from the free product. Since $\bbh_n$ is simply connected, $\pi_1(A)\to \pi_1(X)$ is a group isomorphism. This isomorphism is continuous by the functorality of $\pi_1$. Lastly, the retraction $X\to A$ collapsing $\bbh_n$ to a point induces a retraction $\pi_1(X)\to\pi_1(A)$ of quasitopological groups for which $\pi_1(A)\to \pi_1(X)$ is a section. Thus $\pi_1(A)\to \pi_1(X)$ is a topological embedding and is therefore an isomorphism of quasitopological groups. Since $\pi_1(A)$ is a topological group (Prop. \ref{atopgrp}), so is $\pi_1(X)$.
\end{proof}

Recall that $A_m=D\cup C_1\cup...\cup C_m$. Let $X_m=A_m\vee \bbh_n$ and $Y_m=C_1\cup...\cup C_m\vee \bbh_n$ so that $X_m=Y_m\vee D$. Note that the inclusion $Y_m\to X_m$ is a homotopy equivalence which induces an isomorphism $\pi_r(Y_m)\cong \pi_r(X_m)$ of quasitoplogical groups for all $r\geq 1$. Hence, we may replace $X_m$ with $Y_m$ when necessary.

Consider the finite CW-complex subspaces: $X_{m,k}=A_m\cup \bigcup_{1\leq j\leq k}S_j$ and $Y_{m,k}=C_1\cup ...\cup C_m\cup \bigcup_{1\leq j\leq k}S_j$ so that $X_{m,k}=Y_{m,k}\vee D$. We may view $X_m=\varprojlim_{k}X_{m,k}$ and $Y_m=\varprojlim_{k}Y_{m,k}$ using 
the inverse sequences \[\xymatrix{
... \ar[r] & X_{m,k+1} \ar[r]^-{f_{k+1,k}} &  X_{m,k} \ar[r]^-{f_{k,k-1}} & X_{m,k-1} \ar[r] & ... \\
\ar[r] & Y_{m,k+1} \ar[u]^{\simeq} \ar[r]_-{g_{k+1,k}} &  Y_{m,k} \ar[u]^{\simeq}\ar[r]_-{g_{k,k-1}} & Y_{m,k-1} \ar[u]^{\simeq} \ar[r] & ...
}\]
where the retractions $f_{k+1,k}$ and $g_{k+1,k}$ collapse $S_{k+1}$ to a point. Moreover, the natural retractions $f_{k}:X_m\to X_{m,k}$ and $g_{k}:Y_m\to Y_{m,k}$ induce continuous homomorphisms $\psi:\pi_n(X_m)\to \varprojlim_{k}\pi_n(X_{m,k})$ and $\phi:\pi_n(Y_m)\to \varprojlim_{k}\pi_n(Y_{m,k})$ respectively which make the following diagram commute:
\[\xymatrix{
\pi_n(X_m) \ar[r]^-{\psi} &\varprojlim_{k}\pi_n(X_{m,k})\\
\pi_n(Y_m) \ar[u]^-{\cong} \ar[r]_-{\phi }& \varprojlim_{k}\pi_n(Y_{m,k}) \ar[u]_-{\cong}
}\]
Here, $\pi_n(X_{m,k})$ and $\pi_n(Y_{m,k})$ are discrete groups since $X_{m,k}$ and $Y_{m,k}$ are locally contractible \cite{GHMMthg} and the inverse limit groups have the natural inverse limit topology (as a subspace of the direct product). The vertical morphisms are homeomorphisms and group isomorphisms.

\begin{lemma}
$\psi:\pi_n(X_m)\to \varprojlim_{k}\pi_n(X_{m,k})$ and $\phi:\pi_n(Y_m)\to \varprojlim_{k}\pi_n(Y_{m,k})$ are continuous group isomorphisms.
\end{lemma}

\begin{proof}
As noted above, it suffices to prove the lemma for either $\psi$ or $\phi$; we apply previous results to prove it for $\phi$.

For $n=1$, $Y_m$ is one-dimensional so the injectivity of $\phi$ follows from the $\pi_1$-shape injectivity of one-dimensional spaces \cite{Eda98}.

Suppose $n\geq 2$. Although $Y_m$ is not locally contractible at its basepoint, it is locally simply connected and therefore admits a universal covering space. Let $p:\wt{Y_m}\to Y_m$ and $p_{k}:\wt{Y_{m,j}}\to Y_{m,k}$ be respective the universal covering maps. We view $\wt{Y_{m,k}}$ as the Caley graph on the free group $F(x_1,...,x_m)$ with a copy of $S_1\vee ...\vee S_k$ attached at each vertex. Similarly, we view $\wt{Y_m}$ as the Caley graph on $F(x_1,...,x_m)$ with a copy of $\bbh_n$ attached at each vertex. 

Since $\wt{Y_{m,j+1}}$ is simply connected, the map $g_{j+1,j}\circ p_{j+1}$ induces a retraction $\wt{g}_{j+1,j}: \wt{Y_{m,j+1}}\to  \wt{Y_{m,j}}$ such that $p_j\circ \wt{g}_{j+1,j}=g_{j+1,j}\circ p_{j+1}$. In particular, $\wt{g}_{j+1,j}$ collapses each copy of $S_{j+1}$ to the vertex at which it is adjoined. We now have an inverse sequence of the maps $p_j$ such that $p=\varprojlim_{j}p_j$:\[\xymatrix{
\wt{Y_m}=\varprojlim_{j}\wt{Y_{m,j}} \ar[d]^{p} \ar[r] & ... \ar[r] & \wt{Y_{m,3}} \ar[r]^{\wt{g}_{3,2}} \ar[d]^{p_3} &  \wt{Y_{m,2}} \ar[r]^{\wt{g}_{2,1}} \ar[d]^{p_2} &  \wt{Y_{m,1}} \ar[d]^{p_1} \\
Y_m=\varprojlim_{j}Y_{m,j}  \ar[r] & ... \ar[r] & Y_{m,3} \ar[r]_{g_{3,2}}  &  Y_{m,2} \ar[r]_{g_{2,1}}  &  Y_{m,1}
}\]
The canonical retractions $\wt{g}_j:\wt{Y_m}\to \wt{Y_{m,j}}$ collapse the copies of $S_{j+1},S_{j+2},...$ to the vertex at which they are adjoined and together induce a continuous homomorphism $\wt{\phi}:\pi_n(\wt{Y_m})\to \varprojlim_{j}\pi_n(\wt{Y_{m,j}})$. Since covering maps induce isomorphisms on $\pi_n$, $n\geq 2$ and an inverse limit of isomorphisms is an isomorphism, we see that the vertical morphisms in the lower square of the following diagram are isomorphisms.
\[\xymatrix{
\bigoplus_{w\in F(x_1,...,x_m)}\prod_{k\geq 1}\pi_n(S_k) \ar[d]_-{\cong} \ar[r] & \varprojlim_{k}\bigoplus_{w\in F(x_1,...,x_m)}\prod_{j=1}^{k}\pi_n(S_j)\ar[d]^-{\cong}\\
\pi_n(\wt{Y_m}) \ar[r]^-{\wt{\phi}} \ar[d]_{p_{\#}} & \varprojlim_{k}\pi_n(\wt{Y_{m,k}}) \ar[d]^{\varprojlim_{k}(p_k)_{\#}}\\
\pi_n(Y_m) \ar[r]_-{\phi} & \varprojlim_{k}\pi_n(Y_{m,k})
}\]
Hence, to show that $\phi$ is a group isomorphism it suffices to show that $\wt{\phi}$ is a group isomorphism.

From Remark \ref{finitecaseremark}, we may identify \[\pi_n(\wt{Y_{m,k}})\cong  \bigoplus_{w\in F(x_1,...,x_m)}\prod_{j=1}^{k}\bbz\]with elements represented as sums $\sum_{w}n_w(s_{1}^{w},...,s_{k}^{w})$ where $s_{j}^{k}\in\bbz$ and all but finitely many $n_w\in\bbz$ are zero. Under this identification, the projection map $(\wt{g}_{k+1,k})_{\#}:\pi_n(\wt{Y_{m,k+1}})\to \pi_n(\wt{Y_{m,k}})$ becomes a countable summand
$$\bigoplus_{w\in F(x_1,...,x_m)}\prod_{j=1}^{k+1}\bbz\to \bigoplus_{w\in F(x_1,...,x_m)}\prod_{j=1}^{k}\bbz,\quad \quad
\sum_{w}n_w(s_{1}^{w},...,s_{k}^{w},s_{k+1}^{w})\mapsto \sum_{w}n_w(s_{1}^{w},...,s_{k}^{w})
$$
of the projection map $\prod_{j=1}^{k+1}\bbz\to \prod_{j=1}^{k}\bbz$. By Corollary \ref{infinitetreecorollary}, there is a natural isomorphism \[\pi_n(\wt{X_{m}})\cong \bigoplus_{w\in F(x_1,...,x_m)}\prod_{k\geq 1}\bbz.\]
With this identification, $\wt{\phi}$ becomes identified with the isomorphism 
 \[\bigoplus_{w\in F(x_1,...,x_m)}\prod_{k\geq 1}\bbz\cong \varprojlim_{k}\bigoplus_{w\in F(x_1,...,x_m)}\prod_{j=1}^{k}\bbz\]defined on elements as
$\sum_{w}n_w(s_{1}^{w},s_{2}^{w},...)\mapsto (\sum_{w}n_w(s_{1}^{w},...,s_{k}^{w}))_{k}$.
\end{proof}

\begin{corollary}
$\pi_n(X_m)\cong \pi_n(Y_m)\cong \bigoplus_{\bbn}\prod_{\bbn}\bbz$ as groups.
\end{corollary}

\begin{corollary}\label{xmishausdorff}
As a topological space, $\pi_n(X_m)$ is Hausdorff.
\end{corollary}

\begin{proof}
Any space which continuously injects into a Hausdorff space is Hausdorff. Since an inverse limit of discrete groups is Hausdorff and $\psi:\pi_n(X_m)\to \varprojlim_{j}\pi_n(X_{m,j})$ is a continuous injection by the previous lemma, $\pi_n(X_m)$ is Hausdorff.
\end{proof}

Now for each $m$, we consider the retractions $R_{m+1,m}:X_{m+1}\to X_m$ which send $C_{m+1}$ homeomorphically to $C_{\infty}$. This forms an inverse system for which $X=\varprojlim_{m}X_{m}$. The retractions $R_{m}:X\to X_m$ that collapse each circle of $\bigcup_{k>m}C_k$ to $C_{\infty}$ so that $R_{m+1}\circ R_{m+1,m}=R_m$ induce a canonical continuous homomorphism $\Psi:\pi_n(X)\to \varprojlim_{m}\pi_n(X_m)$.

\begin{lemma}\label{xinjects}
$\Psi:\pi_n(X)\to \varprojlim_{m}\pi_n(X_m)$ is a continuous injection.
\end{lemma}

\begin{proof}
Continuity of $\Psi$ is immediate from the construction. Suppose $\gamma:S^n\to X$ is a map such that $\gamma_m=R_{m}\circ \gamma:S^n\to X_{m}$ is null-homotopic for each $m$. By Lemma \ref{finitelymanlemma}, $\gamma$ has image in $X\backslash \{e_{m_0+1},e_{m_0+2},...\}$ for some $m_0\geq 1$. Since $X\backslash \{e_{m_0+1},e_{m_0+2},...\}$ deformation retracts onto $X_{m_0}$, $\gamma$ is homotopic (by a homotopy in $X$) to a loop $\gamma '$ with image in $X_{m_0}$. Since $\gamma ' = R_{m_0}\circ \gamma '\simeq R_{m_0}\circ \gamma$ in $X_{m_0}$, $\gamma '$ is null homotopic in $X_{m_0}$. Thus $\gamma$ is null-homotopic in $X$.
\end{proof}

\begin{corollary}\label{hausdorff}
$\pi_n(X)$ is Hausdorff.
\end{corollary}

\begin{proof}
By Corollary \ref{xmishausdorff}, $\pi_n(X_m)$ is Hausdorff. Since the inverse limit of Hausdorff spaces is Hausdorff, $\varprojlim_{m}\pi_n(X_m)$ is a Hausdorff group. By Lemma \ref{xinjects}, $\pi_n(X)$ continuously injects into a Hausdorff group and is therefore Hausdorff.
\end{proof}

\begin{remark}
To explicitly identify the isomorphism class of $\pi_n(X)$ as a group, replace $X$ with its locally path-connected coreflection and contract the disk $D$ to obtain the weakly homotopy equivalent space $X'=\bigcup_{m\geq 1}C_m\vee \bbh_n$. By applying Corollary \ref{infinitetreecorollary} to the universal covering space of $X'$, we have group isomorphisms $\pi_n(X)\cong\pi_n(X')\cong \bigoplus_{w\in F(x_1,x_2,...)}\prod_{k\geq 1}\bbz\cong \bigoplus_{\bbn}\prod_{\bbn}\bbz$.
\end{remark}

\section{Discontinuity of the $\pi_1$-action}

Finally, we show that for the $n$-dimensional space $X\subseteq\bbr^{n+1}$ considered in the previous sections, the $\pi_1$-action, denoted here as the function $a:\pi_1(X)\times \pi_n(X)\to\pi_n(X)$, $a(x,y)=x\ast y$, is not continuous. 

\begin{theorem}
For any $n\geq 1$, the natural group action $a:\pi_1(X)\times \pi_n(X)\to \pi_n(X)$ is not continuous.
\end{theorem}

Note that the case $n=1$ indicates that if $X=A\vee \bbh$ is the one-point union of $A$ and the Hawaiian earring, then the action by conjugation in $\pi_1(X)$ is discontinuous.

Recall that $\pi_1(X)$ is freely generated by the set $\{x_1,x_2,...\}$ and $y_j=[\gamma_j]$ is a generator of $\pi_n(S_j)$. Let $C=\{(x_{i}x_{i+1}...x_{i+j})\ast y_{j}^{i}|i,j\geq 1\}$. We show that $C$ is closed in $\pi_n(X)$ but that $a^{-1}(C)$ is not closed in $\pi_1(X)\times \pi_n(X)$. Our choice of $C$ is inspired Fabel's argument in \cite{Fab11CG}. 

\begin{proposition}\label{largeimageprop}
If $\gamma:S^n\to X$ represents $(x_{i}x_{i+1}...x_{i+j})\ast y_{j}^{i}$, then $\bigcup_{k=i}^{i+j}C_k\subset \gamma(S^n)$. In particular, $\{e_i,e_{i+1},...e_{i+j}\}\subseteq \gamma(S^n)$.
\end{proposition}

\begin{proof}
As we have argued previously, $\gamma$ is homotopic to a loop $\delta$ with image in $X_{m_0}\cap Im(\gamma)$ for some $m_0\geq i+j$. Let $\wt{X_{m_0}}$ be the universal covering space of $X_{m_0}$. Again, we identify $\wt{X_{m_0}}$ with the Caley graph of the free group $F(x_1,...,x_{m_0})$ (with vertices identified with elements of $F(x_1,...,x_{m_0})$ and basepoint the identity element $e$) with a copy of the simply connected space $D\vee \bbh_n$ attached at each vertex. Since $S^n$ is simply connected, there is a unique lift $\wt{\delta}:(S^n,(1,0,...,0))\to (\wt{X_m},e)$. Since the Caley graph is uniquely arc-wise connected, in order for $\delta$ to represent $(x_{i}x_{i+1}...x_{i+j})\ast y_{j}^{i}$, the image of $\wt{\delta}$ must contain the unique minimal edge path from the vertex $e$ to the vertex $(x_{i}x_{i+1}...x_{i+j})$ along the edges $[e,x_{i}]$,$[x_{i},x_{i}x_{i+1}]$,...,$[x_{i}x_{i+1}...x_{i+j-1},x_{i}x_{i+1}...x_{i+j}]$. These edges are mapped respectively by the covering map onto $C_{i}$, $C_{i+1}$,...,$C_{i+j}$, so we have $\bigcup_{k=i}^{i+j}C_k\subset \delta(S^n)\subset\gamma(S^n)$.
\end{proof}

\begin{proposition}
$a^{-1}(C)$ is not closed in $\pi_1(X)\times \pi_n(X)$.
\end{proposition}

\begin{proof}
Since $C$ does not contain the identity of $\pi_n(X)$, the identity $(1,1)$ of $\pi_1(X)\times \pi_n(X)$ is not in $a^{-1}(C)$. We show that $(1,1)$ is a limit point of $a^{-1}(C)$. Let $V\times U$ be a basic open neighborhood of $(1,1)$ in $\pi_1(X)\times \pi_n(X)$. Let $q_i:\Omega^i(X)\to\pi_i(X)$ be the natural quotient map. Find an open neighborhood $W$ of $x_0$ such that $\langle S^n,W\rangle \subseteq q_{n}^{-1}(U)$ in $\Omega^{n}(X)$. Since all but finitely many of the spheres $S_j\subset \bbh_n$ lie in $W$ all but finitely many of the generators $\gamma_j:S^n\to X$ have image in $\langle S^n,W\rangle$. It follows that there exists a $J$ such that if $\gamma_{j}^{i}$ is the map $\gamma_j$ concatenated with itself $i$-times, then $\gamma_{j}^{i}\in \langle S^n,W\rangle$ for all $j\geq J$ and all $i\geq 1$. Thus for homotopy classes, we have $y_{j}^{i}\in U$ for all $j\geq J$ and $i\geq 1$.

Now that $J$ is fixed, notice that the sequence $\ell_{i}\cdot \ell_{i+1}\cdot ...\cdot \ell_{i+J}$, $i\geq 1$ of loops in $\Omega(X)$ converges uniformly to the null-homotopic loop $\ell_{\infty}\cdot \ell_{\infty}\cdots \ell_{\infty}$ (with $J+1$-factors). The continuity of $q_1:\Omega(X)\to\pi_1(X)$ now gives $x_{i}x_{i+1}\cdots x_{i+J}\to 1$ as a sequence in $\pi_1(X)$. Thus there is an $I$ such that $x_{I}x_{I+1}\cdots x_{I+J}\in V$. Since we also have $y_{J}^{I}\in U$, it follows that $((x_{I}x_{I+1}...x_{I+J}), y_{J}^{I})\in (V\times U)\cap a^{-1}(C)$.
\end{proof}

\begin{proposition}
$C$ is closed in $\pi_n(X)$.
\end{proposition}

\begin{proof}
Since $q_n:\Omega^n(X)\to \pi_n(X)$, $q_n(\gamma)=[\gamma]$ is quotient, it suffices to show $q_{n}^{-1}(C)$ is closed in the metric space $\Omega^n(X)$ of based maps $S^n\to X$. Suppose $\delta_k\to \delta$ is a convergent sequence in $\Omega^n(X)$ and $q_n(\delta_k)=[\delta_k]\in C$ for all $k$. In particular, suppose $[\delta_k]=(x_{i_k}x_{i_k+1}...x_{i_k+j_k})\ast y_{j_k}^{i_k}$ for sequences of integers $i_k,j_k\in \bbn$. It suffices to show $[\delta]\in C$. To start, we will show that both sequences $i_k$ and $j_k$ must be eventually constant.

Suppose $j_k$ is unbounded. If necessary, replace $j_k$ with a subsequence to ensure that $j_1<j_2<...$. By Proposition \ref{largeimageprop}, we have $\{e_{i_k},e_{i_k+1},...,e_{i_k+j_k}\}\subset \delta_k(S^n)$ for each $k$. Let $R:X\to A$ be the retraction collapsing $\bbh_n$ to the basepoint and set $f_k=R\circ \delta_k$ and $f=R\circ\delta$. We have $f_k\to f$ uniformly and $\{e_{i_k},e_{i_k+1},...,e_{i_k+j_k}\}\subset f_k(S^n)$, which is a contradiction of Lemma \ref{boundinglemma}. Thus $j_k$ must be bounded.

Since $J=\{j_k|k\geq 1\}$ is a finite set, $\bigvee_{j\in J}S_j\subset \bbh_n$ is a finite wedge of n-spheres. In particular, $\pi_n\left(\bigvee_{j\in J}S_j\right)=\prod_{j\in J}\bbz$ is a discrete abelian group freely generated by $\{y_j|j\in J\}$. Since $\delta_k\to \delta$ and $q_n$ is continuous, we have $[\delta_k]\to [\delta]$ in $\pi_n(X)$. Consider the retraction $r:X\to \bigvee_{j\in J}S_j$ which collapses $A$ and all n-spheres $S_j$, $j\notin J$ to the basepoint. By functorality, the homomorphism $r_{\#}:\pi_n(X)\to \pi_n\left(\bigvee_{j\in J}S_j\right)$ is continuous and thus $r_{\#}([\delta_k])=y_{j_k}^{i_k}\to r_{\#}([\delta])$. However, the only sequences that converge in a discrete group are eventually constant. Thus $i_k$ must be eventually constant and, in particular, bounded.

Since both $i_k$ and $j_k$ are bounded, the set $\{[\delta_k]|k\geq 1\}=\{(x_{i_k}x_{i_k+1}...x_{i_k+j_k})\ast y_{j_k}^{i_k}|k\geq 1\}$ is a finite subset of $\pi_n(X)$. By Corollary \ref{hausdorff}, $\pi_n(X)$ is Hausdorff. Since finite subsets of Hausdorff spaces are closed,  $\{[\delta_k]|k\geq 1\}$ is closed in $\pi_n(X)$. Finally, since $[\delta_k]\to [\delta]$, we must have $[\delta]\in \{[\delta_k]|k\geq 1\}\subset C$.
\end{proof}

\end{document}